\documentclass[12pt]{article}

\usepackage{amssymb,amsmath,amsthm}
\usepackage{verbatim}

\newtheorem{thm}{Theorem}
\newtheorem{lemma}{Lemma}
\newtheorem{prop}{Proposition}

\newcommand{\R}{\mathbb{R}}
\newcommand{\Z}{\mathbb{Z}}
\newcommand{\T}{\mathcal{T}}
\newcommand{\U}{\mathcal{U}}

\newcommand{\B}{\mathcal{B}}

\title{Finite Voronoi decompositions of infinite vertex transitive graphs}
\author{Hilary Finucane}
\date{}

\begin{document}

\maketitle

\begin{abstract}
In this paper, we consider the Voronoi decompositions of an arbitrary infinite vertex-transitive graph $G$. In particular, we are interested in the following question: what is the largest number of Voronoi cells that must be infinite, given sufficiently (but finitely) many Voronoi sites which are sufficiently far from each other? We call this number the survival number $s(G)$. 
The survival number of a graph has an alternative characterization in terms of covering, %: we show $s(G)$ equals the number of balls of radius less than $r$ that are needed to cover a sphere of radius $r$ for large $r$, if the centers of the balls must be far from each other and from the center of the sphere. 
which we use to show that $s(G)$ is always at least two.

The survival number is not a quasi-isometry invariant, but it remains open whether finiteness of the $s(G)$ is. We show that all vertex-transitive graphs with polynomial growth have a finite $s(G)$; vertex-transitive graphs with infinitely many ends have an infinite $s(G)$; the lamplighter graph $LL(\Z)$, which has exponential growth, has a finite $s(G)$; and the lamplighter graph $LL(\Z^2)$, which is Liouville, has an infinite $s(G)$.

\end{abstract}

\section{Introduction}

In \cite{competition}, Benjamini introduces a model of competition on an infinite vertex-transitive graph. In this model, a parameter $m \geq 1$ is fixed, two sets $X_0$ and $Y_0$ are initialized to contain vertices $x_0$ and $y_0$, respectively, and then $X_n$ contains all vertices within distance $m$ of $X_{n-1}$ which are not in $Y_{n-1}$, and $Y_n$ contains all vertices within distance 1 of $Y_{n-1}$ which are not in $X_n$. $X := \cup X_n$ and $Y := \cup Y_n$. On $\Z^d$,  $Y$ is finite for all $m\geq 2$ regardless of the starting points $x_0$ and $y_0$. Benjamini shows that on a hyperbolic Cayley graph, for all $m$ there exist $x_0$ and $y_0$ such that $X$ and $Y$ will both be infinite. It is open to determine for which non-hyperbolic Cayley graphs this is true.

In the case that $m=1$, we can adapt the process so that if there exists an $n$ for which a vertex $v$ is a neighbor of both $X_n$ and $Y_n$ but is not contained in either set, then we add $v$ to both $X_{n+1}$ and to $Y_{n+1}$. In this adapted process, $X$ and $Y$ give us a Voronoi decomposition of the graph, where $X$ and $Y$ are the Voronoi cells of $x_0$ and $y_0$, respectively; that is, $X = \{v:d(v,x_0) \leq d(v,y_0)\}$, and $Y = \{v:d(v,y_0) \leq d(v,x_0)\}$. To see this, suppose without loss of generality that $v$ is a vertex with $d(v,x_0) \leq d(v,y_0)$. Then there is a geodesic from $x_0$ to $v$, and each point on the geodesic is at least as close to $x_0$ as to $y_0$. (Otherwise, $v$ would be closer to $y_0$ than $x_0$.) Let $u_i$ be the $i$-th vertex in the geodesic; before step $i$ in the process, $u_i$ cannot be in $Y$ because $d(u_i,y_0) \geq d(u_i,x_0) = i$, so at step $i$, by induction $u_i$ will be added to $X$, and so $v$ will also be added to $X$.

The first question we consider is whether, in the case that $m=1$, $X$ and $Y$ are always both infinite. If there exists an infinite two-sided geodesic passing through $x_0$ and $y_0$, then the answer to this question is clearly yes. However, such a geodesic does not exist for every pair of points in an infinite vertex-transitive graph. Consider for example an infinite ladder with both diagonals added to each square face; two vertices directly across from each other are not both contained in any infinite geodesic. Restricting to Cayley graphs does not solve the problem; there are pairs of vertices not contained in an infinite geodesic in the lamplighter graph $LL(\Z)$ \cite{meier}. In fact, the following question remains open: call two geodesics equivalent if one is contained in a bounded neighborhood of the other. Are there infinitely many non-equivalent two-sided geodesics in any one-ended vertex-transitive graph?

Here, we show that in every infinite vertex-transitive graph, and for every choice $x_0$ and $y_0$ of starting vertices, $X$ and $Y$ must both be infinite (Proposition~\ref{lowerbound}). However, $X$ and $Y$ need not be isometric; in Section~\ref{sec:lowerbound} we describe a simple example due to Gady Kozma (private communication) for which they are not.

We also consider a model in which there are several competing sets, not just two. In particular, we are interested in the maximum $n$ such that for any $k\geq n$ and any set $v_1, \ldots , v_k$ of initial vertices that are pairwise sufficiently far apart, at least $n$ of the resulting Voronoi cells are infinite. This number provides a measure of the size of the underlying graph $G$; we call it the survival number, denoted $s(G)$. The claim above that $X$ and $Y$ are both always infinite is very similar to the claim that $s(G) \geq 2$ for every infinite vertex-transitive $G$; indeed, the two claims follow from a very similar argument.

The survival number has a useful alternative characterization in terms of covering, which we use to prove the above result that $X$ and $Y$ are always both infinite. And though $s(G)$ is not quasi-isometry invariant, it remains open whether finiteness of $s(G)$ is. We show that $s(G)$ is finite for all vertex-transitive graphs with polynomial growth, infinite when $G$ is vertex-transitive and has infinitely many ends, finite when $G$ is the lamplighter group on $\Z$ (which has exponential growth), and infinite when $G$ is the lamplighter group on $\Z^2$ (which is Liouville). 

In Section~\ref{preliminaries}, we define some key terms. In Section~\ref{sec:survival}, we provide a formal definition of the survival number, prove the equivalent definition, and generalize our definition to arbitrary metric spaces. In Section~\ref{sec:lowerbound}, we show that $s(G)$ is at least two for all infinite vertex-transitive graphs, and that this lower bound is achieved in all graphs with linear growth. In Section~\ref{sec:notinvariant}, we show that $s(G)$ is not quasi-isometry invariant. In Section~\ref{sec:infinite}, we prove several results about the finiteness of $s(G)$. In Section~\ref{sec:openquestions}, we provide a list of open problems.

%%%%%%%%%%%%PRELIMINARIES%%%%%%%%%%%%%%%%%%%%%%%%%%%%

\section{Preliminaries}
\label{preliminaries}
A graph $G = (V,E)$ is {\em vertex-transitive} if for all $u, v \in V$, there exists an automorphism $\phi$ of $G$ such that $\phi(u) = v$. We will assume all of our graphs are infinite vertex-transitive graphs.

For $v \in V$, let $B(v,r) = \{u: d(u,v) \leq r\}$ and $\partial B(v,r) = \{u :\ d(u,v) = r\}$. A graph $G$ has {\em growth function} $f$ if $\bigl| B(v,r) \bigr| = f(r)$; $f$ is independent of $v$ by vertex transitivity. We say $G$ has {\em polynomial growth} if $f$ is bounded by a polynomial, and {\em linear growth} if $f$ is bounded by a linear function. A vertex-transitive graph has polynomial growth iff there exists a constant $C$ such that for all $r$, $B(0,r)$ can be covered by $C$ balls of radius $r/2$ \cite{doubling}.

The {\em number of ends} of an infinite vertex-transitive graph is equal to the $\limsup$ as $r \rightarrow \infty$ of the number of connected components in $G-B(v,r)$.

An important class of vertex-transitive graphs are {\em Cayley graphs}. If $\Gamma$ is a group and $S$ is a generating set for $\Gamma$ such that $g \in S$ iff $g^{-1} \in S$, then the Cayley graph of $\Gamma$ has as vertices the elements of $\Gamma$, with an edge between $g_1$ and $g_2$ iff $g_1g_2^{-1} \in S$. 
%%%%%%%%%%%%%%%%%%%%%%%%%%%%%%%%%%%%%%%%%%%%%%%%%%%%%%%%%%%%%%%%%%%%%%%%%%
%%%%%%%%%%THE SURVIVAL NUMBER%%%%%%%%%%%%%%%%%%%%%%%%%
%%%%%%%%%%%%%%%%%%%%%%%%%%%%%%%%%%%%%%%%%%%%%%%%%%%%%%%%%%%%%%%%%%%%%%%%%%

\section{The survival number }
\label{sec:survival}

\subsection{The main definition}
For $T \subset V$ with $|T| < \infty$ and $v \in T$, let 
$$C(v, T) = \{u \in V : d(u,v) \leq d(u, v')\ \forall v' \in T\}.$$ 
In other words, $C(v,T)$ is the set of vertices of $G$ that are as close to $v$ as to any other vertex in $T$; i.e.\ the Voronoi cell of $v$. Because $|V|= \infty$ and $\cup_{v \in T} C(v,T) = V$, at least one of the $C(v,T)$ must be infinite. If $G = \Z$ and $|T|>1$, then two must be infinite.

The survival number of $G$, denoted $s(G)$, is the maximum number of $C(v,T)$ that must be infinite for any $T$ that is sufficiently large and sufficiently spread out. More formally, let 
$$\T(d,n) = \{T \subset V: n \leq |T| < \infty \mbox{ and }d(v_1, v_2)>d\ \forall v_1, v_2 \in T\}$$ 
and
$$I(T) = \bigl|\{v \in T : |C(v,T)| = \infty\}\bigr|.$$
The survival number of $G$ is:
$$s(G) = \sup \{n : \exists d \mbox{ with } I(T) \geq n\ \forall T\in \mathcal{T} (d,n)\}.$$

A more simple definition would remove the requirement that $d(v_1, v_2) > d$ for all $v_1, v_2 \in T$. However, if this requirement is removed, then the resulting number is less robust; for example, it is upper bounded by the degree and so is smaller for the degree-three tree than for $\Z^2$. Having $d(v_1, v_2)>d$ for all $v_1, v_2 \in T$ is a crucial element of the definition of $s(G)$. However, we show below that to check that $s(G) \geq n$, it suffices to check only those $T$ which are of size exactly $n$, rather than of any size greater than $n$.

%%%%%%%%%%%%%%%%%%%EQUIVALENT DEFINITIONS%%%%%%%%%%%%%%%%%%%%%%%%%%%%%%

\subsection{An equivalent definition}
The survival number can also be characterised in terms of covering:

\begin{thm} \label{equiv2} 
$s(G) \geq n$ iff there exists an integer $d$ and an infinite set $R$ of integers such that for any $r \in R$, there is no way to cover $\partial B(0,r)$ with $n-1$ balls of radius less than $r$ whose centers are pairwise at least $d$ apart from each other and from $v$. \end{thm}

Before proving Theorem~\ref{equiv2}, we will prove the following lemma. Let $\U(d,n) = \T(d,n) - \T(d,n+1)$ be the set of subsets of $V$ of size $n$ whose elements are at pairwise distance at least $d$ from each other. 
\begin{lemma} \label{equiv1}
$s(G) = \sup \{n : \exists d \mbox{ with }I(U) = n\ \forall U\in \U(d,n)\}.$
\end{lemma}
\begin{proof}[Proof of Lemma~\ref{equiv1}]
We will show that for any $n$ and $d$, $I(T) \geq n$ for all $T \in \T(d,n)$ iff $I(U) = n$ for all $U \in \U(d,n)$. Since $\U(d,n) \subset \T(d,n)$, $(\Rightarrow)$ is clear. 

To prove $(\Leftarrow)$, suppose $I(U) = n$ for all $U \in \U(d,n)$, but there exists an $T \in \T(d,n)$ with $I(T)<n$. Then let $v_1, \ldots , v_m$ be the vertices in $T$ with $|C(v_i, T)| = \infty$, and let $T' = T - \{v_1, \ldots , v_m\}$; i.e.\ $v \in T'$ iff $|C(v,T)|< \infty$. Let $S$ denote $\cup_{v \in T'}C(v,T)$. Then $S$ is a finite set, and for all vertices $u$ outside of $S$, there exists an $i$ such that $d(u, v_i) < d(u,v)$ for all $v \in T'$.

Let $v_{m+1}, \ldots , v_n$ be $n-m$ arbitrary elements of $T'$, and let $U = \{v_1, \ldots , v_n\}$. Because all but finitely many vertices are strictly closer to one of the $v_1, \ldots , v_m$ than any of the $v_{m+1}, \ldots , v_n$, we have $|C(v_j,U)| < \infty$ for $j = m+1, \ldots, n$. Thus, $I(U) = m < n$, a contradiction.
\end{proof}

\begin{proof}[Proof of Theorem~\ref{equiv2}]
Fix $n$, and suppose that there exists an integer $d$, a vertex $v$, and an infinite set $R$ of integers such that for any $r \in R$, there is no way to cover $\partial B(v,r)$ with $n-1$ balls of radius less than $r$ whose centers are pairwise at least $d$ apart from each other and from $v$. By vertex transitivity, if this holds for one vertex $v$ then it holds for any vertex. Then for each $r \in R$, any $U \in \U(d,n)$, and each $v \in U$, there exists a vertex $u$ with $d(u,v) = r$ and $d(u,v') \geq r$ for all $v' \in U$. Each such $u$ is in $C(v,U)$, so $|C(v,U)| = \infty$. This holds for all $v \in U$, so $s(G) \geq n$. 

Conversely, suppose that for all $d$, all $v$ and all but finitely many $r$, there exist $v_1, \ldots , v_{n-1}$ such that $\{v, v_1, \ldots v_{n-1}\} \in \U(d,n)$ and for every $u \in \partial B(v,r)$ there is an $i$ such that $d(u,v_i) < r$. Then in particular, this holds for one such $r$, so for every vertex $w$ with $d(v,w) \geq r$, the shortest path from $w$ to $v$ must path through a point $u \in \partial B(v,r)$. But $u$ is closer to one of the $v_i$ than to $v$, so $w$ must be closer to $v_i$ than to $v$, so $w$ is not in $C(v,U)$. Thus, $|C(v,U)| < \infty$, so $s(G) < n$.  
\end{proof}

%%%%%%%%%%%%%%%%%%%CONTINUOUS%%%%%%%%%%%%%%%%%%%%%%%

\subsection{A continuous version}

The definition of the survival number of a graph extends to any metric space with a transitive isometry group; the only adaptation needed is to redefine $I(T)$ as $\bigl|\{v \in T : C(v,T)\mbox{ is unbounded}\}\bigr|.$ Theorem~\ref{equiv2} also holds, where $R$ is required to be an unbounded set of positive real numbers, rather than an infinite set of integers.

\begin{prop}\label{Rn} $s(\R^n)=n+1$ \end{prop}

\begin{proof} Borsuk showed that in $\R^n$, a ball (and therefore a sphere) of radius one can be covered by $n+1$ balls of radius strictly less than one \cite{Borsuk}. Let $\epsilon$ be the minimum of the pairwise distances of the centers of these balls to each other and to the origin, and let $R = d/\epsilon$. Then for any $r > R$, a sphere of radius $r$ can be covered by $n+1$ balls whose centers are at distance at least $d$ from each other and from the origin. Lusternik and Schnirelmann proved that $n$ balls do not suffice \cite{LusSch}.
\end{proof}

%%%%%%%%%%%%%%%%%%%%%%%%%%%%%%%%%%%%%%%%%%%%%%%%%%%%%%%%%%%%%%%%%%%%%%%%%%
%%%%%%%%%%%%%%A LOWER BOUND%%%%%%%%%%%%%%%%%%
%%%%%%%%%%%%%%%%%%%%%%%%%%%%%%%%%%%%%%%%%%%%%%%%%%%%%%%%%%%%%%%%%%%%%%%%%%

\section{A lower bound}
\label{sec:lowerbound}

The characterization of Theorem~\ref{equiv2} gives us the following lower bound on $s(G)$:

\begin{prop}\label{lowerbound}Every infinite vertex-transitive graph $G$ has $s(G) \geq 2$. \end{prop}

\begin{proof} Let $v$ be a vertex in $G$. Then there exists a bi-infinite geodesic passing through $v$ \cite{3caret}, so for all $r$ there are two vertices in $\partial B(v,r)$ at distance $2r$ from each other. Since no ball of radius $r-1$ can contain both of these vertices, Theorem~\ref{equiv2} with $d = 0$ shows that $s(G) \geq 2$. \end{proof}

Proposition~\ref{lowerbound} is not equivalent to the claim made in the introduction that $X$ and $Y$ are both always infinite. However, because $d=0$ in the proof of the proposition, this claim follows from the same argument.

So for $T = \{c_1, c_2\}$, $C(v_1, T)$ and $C(v_2, T)$ must both be infinite. However, they are not necessarily isomorphic. The following example is due to Gady Kozma: consider $v_1 = a$, $v_2 = b$, and $T = \{v_1, v_2\}$ in the Cayley graph of \mbox{$F_2 * F_3 = <a, b\ |\ a^2= b^3=1>$}. Then $e$ is in $C(a,T)$ and has one neighbor in $C(a,T)$, but there are no vertices of degree one in $C(b,T)$.

Now, we will show that the lower bound of Theorem~\ref{lowerbound} is achieved in all linear-growth graphs. The following result follows from Gromov's theorem, and has an elementary proof in \cite{3caret}:
\begin{prop}
\label{3caret}
Let $G$ be an infinite vertex-transitive graph with linear growth. Then $G$ contains a bi-infinite geodesic $\gamma$, and there exists a constant $k$ such that every vertex of $G$ is within distance $k$ of $\gamma$.
\end{prop}

We will now use Proposition~\ref{3caret} to prove:

\begin{prop} \label{thm:linear} If $G$ is an infinite vertex transitive graph with linear growth then $s(G) = 2$. \end{prop}

\begin{proof} Let $\phi:\Z \rightarrow G$ be a an isometry from $\Z$ to $\gamma$. Let $v = \phi(0)$. For every vertex $u$, there is an integer $m$ such that $d(u,\phi(m))<k$, and so
\begin{eqnarray*}
\bigl|d(v,u) - |m|\bigr| &=& \bigl|d(v,u) - d(v,\phi(m))\bigr| \\
&\leq& d(u,\phi(m))\\
&<& k
\end{eqnarray*}
So if $d(v,u) = r$ then 
$$m \in [-r-k, -r+k] \cup [r-k,r+k],$$ and 
$$\phi(m) \in B(\phi(-r),k) \cup B(\phi(r),k),$$ so 
$$v \in B(\phi(-r),2k) \cup B(\phi(r),2k).$$ 
Thus, if $r > 2k$, we have that $\partial B(v,r)$ is contained in $B(\phi(-r),r-1) \cup B(\phi(r),r-1)$, and $d(\phi(-r), v), d(\phi(r),v) > r-k$. So for any $d$ we can let $R = \{r \in \Z : r-k \geq d\}$, and apply Theorem~\ref{equiv2} to see that $s(G) \leq 2$. We have $s(G) \geq 2$ from Proposition~\ref{lowerbound}, concluding the proof.
\end{proof}

%%%%%%%%%%%%%%SENSITIVITY%%%%%%%%%%%%%%%%%%

\section{The survival number is not quasi-isometry invariant.}
\label{sec:notinvariant}

Not only is the survival number not a quasi-isometry invariant; in fact, the survival number of a Cayley graph can depend on the generators chosen.

\begin{prop} \label{Z2}Let $G$ denote the Cayley graph of $\Z^2$ with generating set $\{(\pm 1, 0), (0,\pm 1)\}.$ Then  $s(G) = 4$. \end{prop}
\begin{proof} We will use Theorem~\ref{equiv2}. The points $(0,r), (r,0), (0,-r),(-r,0) \in \partial B(0,r)$ but there is no $v$ such that $\partial B(v,r-1)$ covers more than one of these four points. So for all $r$, four balls of radius $r-1$ are needed to cover a sphere of radius $r$. Since four balls centered at these points suffice to cover the sphere, we have $s(\Z^2) = 4$.
\end{proof}

\begin{prop} \label{Z22}Let $G$ denote the Cayley graph of $\Z^2$ with generating set $\{(\pm 1, 0), (0,\pm 1), (1,1), (-1, -1)\}.$ Then  $s(G) = 3$. \end{prop}
\begin{proof} A short calculation shows that for $r$ large enough depending on $d$, 
$$\partial B(0,r) \subset B((d,-d),r-1) \cup B((d,2d), r-1) \cup B((-2d,-d), r-1),$$ so $s(G) \leq 3$. Because $\partial B(0,r)$ contains the points $(0,r)$, $(r,r)$, $(r,0)$, $(0,-r)$, $(-r,-r)$, and $(-r,0)$, no three of which can be covered by a single ball of radius $r-1$, we also get $s(G) \geq 3$. \end{proof}

%%%%%%%%%%%%%%%%%%%%%%%%%%%%%%%%%%%%%%%%%%%%%%%%%%%%
%%%%%%%%%%%%%%%%%INFINITE VS FINITE%%%%%%%%%%%%%%%%
%%%%%%%%%%%%%%%%%%%%%%%%%%%%%%%%%%%%%%%%%%%%%%%%%%%%

\section{Infinite vs.\ finite survival numbers}
\label{sec:infinite}

While $s(G)$ is not quasi-isometry invariant, it remains open whether finiteness of $s(G)$ is quasi-isometry invariant. In this section, we present the following results about the finiteness of $s(G)$. First, we show that any vertex-transitive graph with polynomial growth has finite $s(G)$, and any vertex-transitive graph with infinitely many ends has infinite $s(G)$. We then show that finiteness of $s(G)$ does not depend only on the growth function and that not all Liouville graphs have finite survival number, by proving that $s(LL(\Z))<\infty$ and $s(LL(\Z^2)) = \infty$. (Recall that a graph is called Liouville if it admits no non-constant bounded harmonic functions.)

\begin{thm}\label{poly_ends} If $G$ is an infinite vertex-transitive graph with polynomial growth, then $s(G)< \infty$. If $G$ is a vertex-transitive graph with infinitely many ends then $s(G) = \infty$. \end{thm}

\begin{proof} To prove the first claim, it suffices to show that for all infinite vertex-transitive polynomial-growth graphs, there exists a constant $C$ such that for all $d$ and all $r$, it is possible to cover $\partial B(v,r)$ with at most $C$ balls of radius $r-1$ whose centers are at least $d$ apart from each other and from $v$. We know that there exists a $C$ such that for all $r$ it is possible to cover $B(v,r)$ with $C$ balls of radius $r/2$ \cite{doubling}. Intuitively, we will argue that the extra freedom we get from using balls of radius $r-1$ instead of $r/2$ allows us to move the centers of the balls so that they are all sufficiently far from each other. If there is a ball that cannot be moved in this way, then it must already be covered by other balls, so it can be removed.

Now, fix $d$ and $r$, and suppose that $B(v,r) \subset B(u_1, r/2) \cup \ldots \cup B(u_C,r/2)$. Let $\B = \{B(u_1,r/2), \ldots , B(u_C,r/2)\}$, and wlog suppose that $v  = u_1$. We will modify $\B$ via the following iterative procedure: for $i = 2, \ldots, C$, if there is a point $w \in B(u_i, r/2-1)$ that is at distance at least $d$ from the centers of all the other balls in $\B$, let $\B' = \bigl(\B - B(u_i,r/2)\bigr) \cup B(w, r-1)$. Otherwise, let $\B' = \B - B(u_i, r/2)$. Note that in the first case, $B(u_i, r/2) \subset B(w, r-1)$, and in the second case, $B(u_i, r/2) \subset \B - B_i$, so in both cases $\B'$ covers all the vertices covered by $\B$. Let $\B = \B'$. After repeating this process for $i = 2, \ldots , C$, replace $B(u_1,r/2)$ with $B(u_1, r-1)$.

Now, $\B$ is a set of balls of radius $r-1$ that cover $B(v,r)$, whose centers are all at least $d$ apart from each other, one of which is centered at $v$. The last step of the process is to remove $B(v,r-1)$ from $\B$; note that $B(v, r-1) \cap \partial B(v,r) = \emptyset$, so $\B$ still covers $\partial B(v,r)$. This gives a cover of $\partial B(v,r)$ by balls of radius $r-1$ around vertices at distance at least $d$ from each other and from $v$.

Now we will prove the second statement; assume $G$ is an infinite vertex-transitive graph with infinitely many ends.
Recall that $\U(d,n)$ denotes the set of subsets of $V$ of size $n$ whose elements are at pairwise distance at least $d$ from each other. To show $s(G) = \infty$ it suffices to find, for any $n$, a $d$ such that for all $U \in \U(d,n)$, $I(U) \geq n$. 

Fix $n$, let $v$ be a vertex, choose $r$ to be large enough that there are more than $n-1$ infinite connected components in $G - B(v,r)$, and let $d>2r$. Then for $U \in \U(d,n)$ and $v \in U$, there is an infinite connected component $A$ of $G-B(v,r)$ that does not contain any vertices in $U$. Let $w$ be a vertex in $A$. Any path from $w$ to a $v' \in U$, $v' \neq v$ must pass through $B(v,r)$. The distance from $w$ to $B(v,r)$ is at least $d(w,v) - r$, and the distance from $B(v,r)$ to $v'$ is greater than $r$ because $d(v,v')>2r$, so $d(w,v') > d(w,v) - r + r = d(w,v)$. Thus, $|C(v,U)| = \infty$. Since $U$ and $v$ were arbitrary, $s(G) \geq n$. This holds for all $n$, so $s(G) = \infty$.
\end{proof}

The following theorem concerns the lamplighter graph $LL(G)$, where $G$ is an underlying graph. Each vertex of $LL(G)$ is made up of a configuration $x$ of lit and unlit lamps---i.e.\ an assignment $x(v) \in \{0,1\}$ to each vertex $v$ of $G$ such that $x(v) = 1$ for finitely many $v$---together with a location of the lamplighter---i.e.\ a single vertex $v_L$ of $G$. A vertex $(x,v_L)$ in $LL(G)$ is connected to another vertex $(x', v'_L)$ if $d(v_L, v_L') = 1$ and $x(v) = x'(v)$ for all $v \notin \{v_L, v'_L\}$. In other words, at each step, the lamplighter is allowed to change or not change the state of the current lamp, traverse an edge in $G$, and then change or not change the new current lamp. For all $d$, $LL(\Z^d)$ has exponential growth and is amenable, and for $d \leq 2$, $LL(\Z^d)$ is Liouville. The origin in $LL(\Z^d)$ is the all-zeros configuration with the lamplighter at zero.

\begin{thm} $s(LL(\Z)) < \infty$ and $s(LL(\Z^2)) = \infty$. \end{thm}
\label{lamplighter}
\begin{proof}
We begin with $LL(\Z)$. Each element of $\partial B(0,r)$ in $LL(\Z)$ can be obtained by a sequence of $r$ moves that start at the origin, and when $r$ is large enough, during these $r$ moves the lamplighter must visit either all of the integers $1, \ldots , 4d$, or all of $-1, \ldots, -4d$; there is no geodesic path in $LL(\Z)$ for which the lamplighter stays in an interval smaller than half the length of the path. Thus, each vertex in $\partial B(0,r)$ in $LL(\Z)$ is within distance $r-1$ of one of the following vertices, depending on what $v_L$ and $x(0)$ are after the first step, and whether integers $1, \ldots 4d$ or $-1, \ldots , -4d$ were visited. (Let $x(v) = 0$ when not specified.)
\begin{itemize}
\item $v_L = 1$, $x(0) = 0$, $x(1) = \cdots = x(d) = 1$;
\item $v_L = 1$, $x(0) = 0$, $x(-1) = \cdots = x(-d) = 1$;
\item $v_L = 1$, $x(0) = 1$, $x(d+1) = \cdots = x(2d) = 1$;
\item $v_L = 1$, $x(0) = 1$, $x(-d-1) = \cdots = x(-2d) = 1$;
\item $v_L = -1$, $x(0) = 0$, $x(2d+1) = \cdots = x(3d) = 1$;
\item $v_L = -1$, $x(0) = 0$, $x(-2d-1) = \cdots = x(-3d) = 1$;
\item $v_L = -1$, $x(0) = 1$, $x(3d+1) = \cdots = x(4d) = 1$;
\item $v_L = -1$, $x(0) = 1$, $x(-3d-1) = \cdots = x(4d) = 1$.
\end{itemize}
These eight vertices are also at distance $d$ from each other and from the origin. 

Now we prove the second claim. Given any geodesic path $\gamma$ in $\Z^2$ from $0$ to $u$ for some $u \in \partial B(0,r)$, the vertex $(x_\gamma,u)$ of $LL(\Z^2)$ is in $\partial B(0,r) \subset LL(\Z^2)$, where $x_\gamma(v) = 1$ for $v \in \gamma$ and 0 otherwise.

Fix $n$, and let $m$ be such that $\partial B(0,m)\geq n$ in $\Z^2$. Then for $r \geq m$, it is possible to choose a set $\Gamma$ of $n$ geodesic paths from $0$ to $\partial B(0,r)$ in $\Z^2$ that are pairwise disjoint outside of $B(0,m-1)$.  If $n-1$ balls of radius $r-1$ cover $\partial B(0,r)$ in $LL(\Z^2)$, then one of them, say $B((y,w),r-1)$, must cover some pair $(x_\gamma , u)$ and $(x_{\gamma'}, u')$, with $\gamma, \gamma' \in \Gamma$. Since $\gamma$ and $\gamma'$ are disjoint on $\partial B(0,m)$, the path in $LL(\Z^2)$ from $(y,w)$ to either $(x_\gamma, u)$ or $(x_\gamma', u')$ must involve the switching of a lamp at distance $m$ from 0. So the lamplighter must walk from $w$ to $\partial B(0,m)$ to $\partial B(0,r)$ in $r-1$ steps, a path of length at least $|w|+r - 2m$. For this to be possible, $|w| < 2m$. 

When $r$ is large enough, we can further require the following of $\Gamma$: there exists an integer $M>m$ such that for any two paths $\gamma$ and $\gamma'$ in $\Gamma$ with endpoints $u$ and $u'$ in $\partial B(0,r)$, for any vertex $w$ with $|w| < 2m$ and any $v$ with $|v|>M$, either the path from $w$ to $v$ to $u$ has length greater than $n-1$, or the path from $w$ to $v$ to $u'$ has length greater than $n-1$. So now, choose $d$ large enough so that if $|w| < 2m$, then for $|(y,w)|>d$ to hold, $y(v)$ must be one for some $|v| \geq M$. Let $p$ denote the shortest path from $w$ to $v$ to $u$, and $p'$ denote the shortest path from $w$ to $v$ to $u'$. If $p$ and $p'$ have the same length, then this length is longer than $r-1$, and since $|v| \geq m$, $v$ is either not contained in $\gamma$ or not contained in $\gamma'$; we will say it is not contained in $\gamma$. If $p$ and $p'$ have different lengths, suppose without loss of generality that $p$ is longer than $p'$. Then again we have $|p| > r-1$, and $d(v,u) > d(v,u')$, and since this does not hold for any vertex on $\gamma$, $v$ must not be on $\gamma$. Since $v$ is not on $\gamma$ and $y(v) = 1$, any path from $(y,w)$ to $(x_\gamma, u)$ must involve the lamplighter passing through $v$ and switching its state to 0. But because $|p| > r-1$, this cannot be done in $r-1$ moves.

Thus, for every $n$ there exists a $d$ such that for all $r$ large enough, $\partial B(0,r)$ cannot be covered by $n-1$ balls of radius $r-1$ whose centers are at distance at least $d$ from 0.

\end{proof}

In \cite{stathyp}, the sprawl of an infinite Cayley graph $G$ is defined as
$$E(G) = \lim_{n \rightarrow \infty} \frac{1}{|\partial B(0,n)|^2} \sum_{x,y \in \partial B(0,n)} \frac{1}{n}d(x,y).$$
The maximal possible value of $E(G)$ is two, and is achieved by hyperbolic groups; a graph $G$ with $E(G)= 2$ is called {\em statistically hyperbolic}. 

If for all $n$ there exist infinitely many $r$ for which there are $n$ elements of $\partial B(0,r)$ which are at pairwise distance $2r$ from each other, then $s(G) = \infty$. If a graph is statistically hyperbolic, then the asymptotic average distance of two elements of $\partial B(0,r)$ is $2r$, so it seems natural to wonder whether statistical hyperbolicity implies $s(G) = \infty$. However, $LL(\Z)$ is shown in \cite{stathyp} to be statistically hyperbolic, so Theorem~\ref{lamplighter} shows that statistical hyperbolicity does not imply $s(G) = \infty$. 

%%%%%%%%%%%%%%%%%%OPEN QUESTIONS%%%%%%%%%%%%%%%%%%%%

\section{Open Problems}
\label{sec:openquestions}

A number of interesting open questions remain:
\begin{enumerate}

\item Is finiteness of $s(G)$ a quasi-isometry invariant? 

\item If $H$ is a subgroup or a quotient of $G$, is $s(H) \leq s(G)$?

\item Is the survival number related to the minimum rate of divergence of geodesics? In particular, does hyperbolicity imply $s(G) = \infty$? 

\item Does non-Liouville imply $s(G) = \infty$? 

If for all $n$ there exist infinitely many $r$ for which there are $n$ elements of $\partial B(v,r)$ which are at pairwise distance $2r$ from each other, then $s(G) = \infty$. If $X_i$ is the $i$-th position of the simple random walk on a Cayley graph starting at 0, then 
$$c := \lim_{n \rightarrow \infty} \frac{1}{n} |X_n|$$ 
exists almost surely \cite{LLN}. The random walk is called ballistic if $c>0$; a graph is non-Liouville iff the simple random walk is ballistic \cite{KaimWoess}. So on a non-Liouville graph, after walking $(1/c)r$ steps starting at $x$ for large $r$, the vertex $y$ reached has $d(x,y) = r+o(r)$ almost surely. If $z$ is a vertex reached by an independent copy of the simple random walk starting at $x$, then $z$ is also at distance $r+o(r)$ from $x$ almost surely. By reversability of the simple random walk and vertex transitivity, we can couple so that $z$ was reached after a simple random walk of $(2/c)r$ steps from $y$, ensuring that $d(z,y) = 2r+o(r)$ almost surely. This gives us an arbitrarily large set of vertices at distance $r+o(r)$ from $x$ and $2r+o(r)$ from each other. 

However, the covering condition is not robust enough for this fact to be directly useful. For example, in $LL(\Z)$, $\partial B(0,r)$ can be covered by finitely many balls of radius $r-1$ (Theorem~\ref{lamplighter}), but because $LL(\Z)$ has exponential growth, $|\partial B(0, r+\omega(1))| = |B(0,r-1)| \cdot \omega(1)$, and so $\partial B(0,r+\omega(1))$ cannot be covered by a constant number of balls of radius $r-1$ (recall that a function $f$ is in $\omega(1)$ if $\lim_{x \rightarrow \infty} f(x) = \infty$). Thus, it remains open whether all non-Liouville graphs have an infinite survival number. (If a non-Liouville Cayley graph admits a symmetric random walk such that the harmonic measure on the sphere is uniform or close to uniform, however, then this argument shows that the graph is statistically hyperbolic.)

\item More generally, for which infinite vertex-transitive graphs $G$ does $s(G) = \infty$?

\item Does $s(G) = 2$ imply $G$ has linear growth?

\item The survival number is the minimum number of Voronoi cells that must be infinite; what is the maximum number of Voronoi cells that can be infinite? Is this supremum infinity for all non-linear vertex-transitive graphs? 

\item Recall that 
$$\T(d,n) = \{T \subset V: n \leq |T| < \infty \mbox{ and }d(v_1, v_2)>d\ \forall v_1, v_2 \in T\}$$ 
and
$$I(T) = \bigl|\{v \in T : |C(v,T)| = \infty\}\bigr|,$$ 
and let 
$$t(G) = \lim_{n \rightarrow \infty} \min\{I(T) : |T| = n\}.$$ 
Does $t(G) = s(G)$? If not, what are the properties of $t(G)$?

\end{enumerate}

\section{Acknowledgements}
Thanks to Itai Benjamini for suggesting this definition of survival number and for many helpful conversations. Thanks also to Gady Kozma, Yakir Reshef, Romain Tessera, and Adam Timar for helpful comments.

\bibliographystyle{plain}
\bibliography{survival}

\end{document}